\documentclass[11pt]{amsart}

\usepackage{amsmath,amsthm}

\usepackage{amssymb}

\usepackage{pdfsync}

\newcommand{\R}{{\mathbb R}}       
\newcommand{\Z}{{\mathbb Z}}       
\newcommand{\DD}{{\mathcal D}}
\newcommand{\HH}{{\mathcal H}}

\newcommand{\NN}{{\mathcal N}}

\newcommand{\CH}{{\mathcal Ch}}

\newcommand{\diam}{{\rm diam}}
\newcommand{\dist}{{\rm dist}}

\newcommand{\rf}[1]{{(\ref{#1})}}
\newcommand{\supp}{\operatorname{supp}}

\newcommand{\vphi}{{\varphi}}
\newcommand{\ve}{{\varepsilon}}
\newcommand{\vv}{{\vspace{2mm}}}
\newcommand{\vvv}{{\vspace{3mm}}}
\newcommand{\wt}[1]{{\widetilde{#1}}}

\newcommand{\wfi}{{\wt \vphi_r}}
\newcommand{\wf}{{\wt \vphi}}

\newtheorem{theorem}{Theorem}[section]
\newtheorem*{theorema*}{Theorem A}
\newtheorem*{theoremb*}{Theorem B}
\newtheorem*{theoremc*}{Theorem C}
\newtheorem*{theoremd*}{Theorem D}
\newtheorem{lemma}[theorem]{Lemma}

\newtheorem{coro}[theorem]{Corollary}
\newtheorem{propo}[theorem]{Proposition}

\theoremstyle{definition}

\theoremstyle{remark}

\numberwithin{equation}{section}

\begin{document}

\begin{abstract}
Let $E$ be a set in $\R^d$ with finite $n$-dimensional Hausdorff measure $\HH^n$ such
that $\liminf_{r\to0}r^{-n}\HH^n(B(x,r)\cap E)>0$ for $\HH^n$-a.e.\ $x\in E$.
In this paper it is shown that $E$ is $n$-rectifiable if and only if 
$$\int_0^1 \left|\frac{\HH^n(B(x,r)\cap E)}{r^n} - \frac{\HH^n(B(x,2r)\cap E)}{(2r)^n}\right|^2\,\frac{dr}r< \infty
\quad
\mbox{for $\HH^n$-a.e.\ $x\in E$,}$$
and also if and only if
$$
\lim_{r\to0}\left(\frac{\HH^n(B(x,r)\cap E)}{r^n} - \frac{\HH^n(B(x,2r)\cap E)}{(2r)^n}\right)
 =0\quad
\mbox{for $\HH^n$-a.e.\ $x\in E$.}$$
Other more general results involving Radon measures are also proved.
\end{abstract}

\title{Rectifiability via a square function and Preiss' theorem}

\author{Xavier Tolsa and Tatiana Toro}
\address{Xavier Tolsa, ICREA/Universitat Aut\`onoma de Barcelona, Bellaterra 08913, Catalonia}
\address{Tatiana Toro, Department of Mathematics, Box 354350, University of Washington, Seattle, WA 98195-4350}

\thanks{X.T.\ was partially supported by 
the ERC grant 320501 of the European Research
Council (FP7/2007-2013) 
and by the
grants
2009SGR-000420 (Catalonia) and MTM-2010-16232 (Spain). T.T.\ was partially supported by an NSF grant DMS-0856687, the Simons Foundation grant \# 228118 and the Robert R. \& Elaine F. Phelps Professorship in Mathematics
}

\maketitle

\section{Introduction}

A set $E\subset \R^d$ is called $n$-rectifiable if there are Lipschitz maps
$f_i:\R^n\to\R^d$, $i=1,2,\ldots$, such that 
\begin{equation}\label{eq001}
\HH^n\biggl(\R^d\setminus\bigcup_i f_i(\R^n)\biggr) = 0,
\end{equation}
where $\HH^n$ stands for the $n$-dimensional Hausdorff measure. On the other hand, using Mattila's definition 
\cite[Definition 16.6]{Mattila-llibre} we say that a 
a Radon measure $\mu$ on $\R^d$ is called $n$-rectifiable if $\mu$ vanishes out of a rectifiable
set $E\subset\R^d$ and moreover $\mu$ is absolutely continuous with respect to $\HH^n|_E$.

One of the main objectives
of geometric measure theory consists in characterizing $n$-rectifiable sets and measures in 
different ways. For instance, there exist characterizations in terms of the existence of
approximate tangent $n$-planes, in terms of the existence of densities, or in terms of the
size of projections. These results stem from the works for the case $n=1$ in the plane by Besicovitch  at beginning of the last century
and have been extended to the whole range of dimensions in the space by different authors.
See for example the book by Mattila \cite{Mattila-llibre} for more details about this beautiful theory.

Preiss' theorem \cite{Preiss} is one of the great landmarks of geometric measure theory. This
asserts that a Radon measure $\mu$ on $\R^d$ is $n$-rectifiable if and only if the density
\begin{equation}\label{teopreiss}
\Theta^n(x,\mu)=\lim_{r\to0}\frac{\mu(B(x,r))}{r^n}
\end{equation}
exists and is positive for $\mu$-a.e.\ $x\in\R^d$. In \cite{DKT} and \cite{PTT} the authors proved that 
the rate of convergence of the density ratio to its limit yields additional information over the regularity of the support of the measure.
In this paper we will prove two variants of Preiss' result. One can be considered as a square function version
of Preiss' theorem. The other shows that the condition on the existence of the limit 
\rf{teopreiss} can be weakened considerably. 
 To state our results we need first to introduce some additional notation.

Given a Radon measure $\mu$ and $x\in\R^d$ we denote
$$\Theta^{n,*}(x,\mu)=\limsup_{r\to0}\frac{\mu(B(x,r))}{r^n},
\qquad 
\Theta^{n}_*(x,\mu)=\liminf_{r\to0}\frac{\mu(B(x,r))}{r^n}.
$$
These are the upper and lower $n$-dimensional densities of $\mu$ at $x$. If they coincide,
they are denoted by $\Theta^{n}(x,\mu)$. In the case when $\mu=\HH^n|_E$ for some set
$E\subset\R^d$, we also write $\Theta^{n,*}(x,E)$, $\Theta^{n}_*(x,E)$, $\Theta^{n}(x,E)$
instead of $\Theta^{n,*}(x,\HH^n|_E)$, $\Theta^{n}_*(x,\HH^n|_E)$, $\Theta^{n}(x,\HH^n|_E)$,
respectively.

The main result of this paper reads as follows.

\begin{theorem}\label{teomain}
Let $\mu$ be a Radon measure in $\R^d$ such that $0< \Theta_*^n(x,\mu)\leq \Theta^{n,*}(x,\mu)<\infty$
for $\mu$-a.e.\ $x\in\R^d$. The following are equivalent:
\vspace{1mm}
\begin{itemize}
\item[(a)] $\mu$ is $n$-rectifiable.
\vspace{3mm}
\item[(b)] 
$\displaystyle
\int_0^1 \left|\frac{\mu(B(x,r))}{r^n} - \frac{\mu(B(x,2r))}{(2r)^n}\right|^2\,\frac{dr}r< \infty
\;$ for $\mu$-a.e.\ $x\in\R^d$.
\vspace{3mm}
\item[(c)] 
$\displaystyle
\lim_{r\to0}\left(\frac{\mu(B(x,r))}{r^n} - \frac{\mu(B(x,2r))}{(2r)^n}\right) =0
\;$ for $\mu$-a.e.\ $x\in\R^d$.
\end{itemize}
\end{theorem}

The following is an immediate consequence.

\begin{coro}\label{coromain}
Let $E\subset\R^d$ be a Borel set with $\HH^n(E)<\infty$ such that
$\Theta_*^n(x,E)>0$
for $\HH^n$-a.e.\ $x\in E$. The following are equivalent:
\vspace{1mm}
\begin{itemize}
\item[(i)] $E$ is $n$-rectifiable.
\vspace{3mm}
\item[(ii)] 
$\displaystyle
\int_0^1 \left|\frac{\HH^n(B(x,r)\cap E)}{r^n} - \frac{\HH^n(B(x,2r)\cap E)}{(2r)^n}\right|^2\,\frac{dr}r< \infty
\;$ for $\HH^n$-a.e.\ $x\in E$.

\vspace{3mm}
\item[(iii)] 
$\displaystyle
\lim_{r\to0}\left(\frac{\HH^n(B(x,r)\cap E)}{r^n} - \frac{\HH^n(B(x,2r)\cap E)}{(2r)^n}\right)
 =0
\;$ for $\HH^n$-a.e.\ $x\in E$.
\end{itemize}
\end{coro}

Notice that the fact that 
$\HH^n(E)<\infty$ implies that the condition $\Theta^{n,*}(x,E)<\infty$ is satisfied 
for $\HH^n$-a.e.\ $x\in E$  
(see \cite[Theorem 6.2]{Mattila-llibre}, for example).

Some remarks are in order. First we mention that the equivalence (a)$\Leftrightarrow$(b) in Theorem \ref{teomain} is a pointwise version
of a related result which characterizes the so called uniform rectifiability, which
was recently obtained in \cite{CGLT}.  The implication 
whose proof requires more effort in this paper is (a)$\Rightarrow$(b). To prove this we will introduce a square function operator and we
will show that it is bounded from the space of finite real measures on $\R^d$ to 
$L^{1,\infty}(\mu)$, by using Calder\'on-Zygmund techniques. 
On the other hand, we will obtain (b)$\Rightarrow$(a) by combining some of the results from \cite{CGLT} with others from Preiss regarding uniform measures, by using ``tangent measure technology''. The implication (c)$\Rightarrow$(a) follows by similar arguments. Notice,
by the way, that the statement in (c) looks much weaker than the $\mu$-a.e.\ existence of the
limit \rf{teopreiss}. So
(c)$\Rightarrow$(a)
can be considered as a strengthening of Preiss' theorem.

The implication (c)$\Rightarrow$(a) in Theorem \ref{teomain} does not hold if 
one replaces the assumption that $ \Theta_*^n(x,\mu)>0$ by $\Theta^{n,*}(x,\mu)>0$ $\mu$-a.e.
Indeed, Preiss has constructed in \cite[5.8-5.9]{Preiss} a measure $\mu$ in the plane 
which is purely $1$-unrectifiable (i.e.\ vanishes on any $1$-rectifiable set and so (a) 
fails for $\mu$)
such that,
for $\mu$-a.e.\ $x\in \R^2$, 
 $0< \Theta^{1,*}(x,\mu)<\infty$ and all tangent
measures at $x$ are $1$-flat (see Section \ref{sec6} for the definition of tangent and flat measures). It is easily seen that this fact implies that $\mu$ satisfies (c).

On the other hand, we do not know if the assumption that $0<\Theta^{n,*}(x,\mu)<\infty$ $\mu$-a.e.
suffices for the validity of the implication (b)$\Rightarrow$(a) in
Theorem \ref{teomain}. If this were true, then Corollary \ref{coromain}
would assert that, given $E\subset\R^d$ with $\HH^n(E)<\infty$, 
$E$ being $n$-rectifiable is equivalent to (ii).

Finally it is worth pointing that there are other characterizations of $n$-rectifiability involving square functions in the literature. Most of them involve the so called $\beta$ coefficients of Peter Jones
\cite{Jones}. See for example Pajot's theorem \cite[Theorem 26]{Pajot} for a result in the
spirit of Theorem  \ref{teomain}, with $\beta$ coefficients instead of densities.

In this paper the letter $c$ stands
for some constant which may change its value at different
occurrences. The notation $A\lesssim B$ means that
there is some fixed constant $c$ such that $A\leq c\,B$,
with $c$ as above. Also, $A\approx B$ is equivalent to $A\lesssim B\lesssim A$.


\section{Preliminaries}\label{secprelim}

\subsection{AD-regularity and uniformly rectifiable measures}

A measure $\mu$ is called $n$-AD-regular if there exists some
constant $c_0>0$ such that
$$c_0^{-1}r^n\leq \mu(B(x,r))\leq c_0\,r^n\quad \mbox{ for all $x\in
\supp(\mu)$ and $0<r\leq \diam(\supp(\mu))$.}$$

A measure $\mu$ is  uniformly  $n$-rectifiable if it is 
$n$-AD-regular and
there exist $\theta, M >0$ such that for all $x \in \supp(\mu)$ and all $r>0$ 
there is a Lipschitz mapping $g$ from the ball $B_n(0,r)$ in $\R^{n}$ to $\R^d$ with $\text{Lip}(g) \leq M$ such that$$
\mu (B(x,r)\cap g(B_{n}(0,r)))\geq \theta r^{n}.$$
In the case $n=1$, $\mu$ is uniformly $1$-rectifiable if and only if $\supp(\mu)$ is contained in a rectifiable curve $\Gamma$ in $\R^d$ such that the arc length measure on $\Gamma$ is $1$-AD-regular.
A set $E\subset\R^d$ is called uniformly $n$-rectifiable if $\HH^n|_E$ is uniformly  $n$-rectifiable.

The notion of uniform rectifiability was introduced by David and Semmes \cite{DS1}, \cite{DS2}.
In these works they showed that a big class of singular singular integrals with odd kernel
is bounded in $L^2(\mu)$ if $\mu$ is uniformly rectifiable. See \cite{ntov} for a recent related result
in the converse direction involving the $n$-dimensional Riesz transforms.

\subsection{Dyadic cubes}

In the case when $\mu$ is an $n$-AD-regular measure in $\R^d$ we will use the David lattice $\DD$ of ``cubes'' associated with $\mu$ (see \cite[Appendix 1]{David}, for example). 
Suppose for simplicity that $\mu(\R^d)=\infty$ (and so $\diam(\supp\mu)=\infty$).
Then one has a disjoint union
$\DD=\bigcup_{j\in\Z}\DD_j$ and each set $Q\in\DD_j$, which is called a cube, is a Borel subset of
$\supp(\mu)$ which satisfies 
$\mu(Q)\approx 2^{-jn}$ and $\diam(Q)\approx2^{-j}$. In fact, we will assume that
$$c^{-1}2^{-j}\leq \diam(Q)\leq 2^{-j}.$$
We set $\ell(Q):=2^{-j}$, and we call this the side length of $Q$. 
For $R \in \DD$, we denote by $\DD(R)$ the family of all cubes $Q \in \DD$ which are contained in $R$. In the case when $\mu(\R^d)<\infty$ and $\diam(\supp(\mu))\approx2^{-j_0}$, then $\DD=\bigcup_{j\geq j_0}
\DD_j$. The other properties of the lattice $\DD$ are the same as in the previous case.

If $Q\in\DD_j$, we write $J(Q)=j$. That is, $J(Q)$ is the generation of $\DD$ to which $Q$ belongs.
On the other hand, we say that $Q,Q'\in\DD$ are neighbors if they belong to the same generation 
(which is equivalent to saying $\ell(Q)=\ell(Q')$) and moreover there exist $x\in Q$ and $x'\in Q'$
such that $|x-x'|\leq \ell(Q)$. We denote the collection of neighbors of $Q$ by $\NN(Q)$. Notice that
$Q\in\NN(Q)$.

For $Q\in\DD$, we denote by $\CH(Q)$ the family of dyadic cubes contained in $R$ with side length
equal to $\ell(Q)/2$. These are the so called children of $Q$.

We will call ``true cubes'' the usual cubes in $\R^d$, to distinguish them from the cubes from $\DD$.

\subsection{The dyadic martingale}

Suppose again that $\mu$ is $n$-AD-regular, and let $\DD$ be the associated dyadic lattice. 
Given $f\in L^1_{loc}(\mu)$ and $Q\in\DD$, we denote  by $m_Qf$ the mean of $f$ on $Q$ with respect to $\mu$. That is,
$$m_Qf = \frac1{\mu(Q)}\int_Q f\,d\mu.$$
 Then we define
$$\Delta_Q f = \sum_{P\in\CH(Q)} \chi_P\,(m_Pf - m_Q f).$$ 
The functions $\Delta_Qf$, $Q\in\DD$, are orthogonal, and it is easy to check that
$$\|f\|_{L^2(\mu)}^2 = \sum_{Q\in\DD}\|\Delta_Q f\|_{L^2(\mu)}^2.$$
For every $Q\in\DD$, we also have
\begin{equation}
 f\,\chi_Q=  (m_Qf)\,\chi_Q + \sum_{P\in\DD(Q)}\Delta_P f,
\label{eqmartingale}
\end{equation}
with the sum converging in $L^2(\mu)$.


\section{The square function operator}\label{secqu}

To prove
 Theorem \ref{teomain} we will first show (a)$\Rightarrow$(b).
 To this end, it is convenient to introduce the
following operator $T$.
Given a real Radon measure $\nu$ on $\R^d$ and $x\in\R^d$, we set
$$T\nu(x) =
\left(\int_0^\infty \left|\frac{\nu(B(x,r))}{r^n} - \frac{\nu(B(x,2r))}{(2r)^n}\right|^2\,\frac{dr}r\right)^{1/2}.$$
Notice that $T$ is a sublinear operator.
For a positive Borel measure $\mu$ on $\R^d$ and a given function $f\in L^1_{loc}(\mu)$, 
we also write 
$$T_\mu f(x) = T(f\mu)(x).$$
That is,
$$T_\mu f(x) = \left(\int_0^\infty \left|\frac{(f\mu)(B(x,r))}{r^n} - \frac{
(f\mu)(B(x,2r))}{(2r)^n}\right|^2\,\frac{dr}r\right)^{1/2},$$
where
$(f\mu)(A) = \int_Af\,d\mu$ for any set $A\subset \R^d$.

It is easy to check that 
to prove that $T\mu(x)<\infty$ for $\mu$-a.e.\ $x\in\R^d$ if $\mu$ is $n$-rectifiable
we may assume that $\mu$ is compactly supported and thus finite.
We will show that if  
$\Gamma$ is an $n$-dimensional Lipschitz graph in $\R^d$ (and more generally, a
uniformly $n$-rectifiable set), then
$T$ is bounded from the space of finite real Borel measures on $\R^d$, denoted by $M(\R^d)$, to
$L^{1,\infty}(\HH^n|_\Gamma)$. That is, there exists some constant $c>0$ such that 
\begin{equation}\label{eqdeb0}
\HH^n\bigl(\{x\in\Gamma: \,T\nu(x)>\lambda\}\bigr) \leq c\,\frac{\|\nu\|}\lambda\quad
\mbox{\;for
all $\nu\in M(\R^d)$ and  $\lambda>0$.}
\end{equation}

To see that $T\mu(x)<\infty$ for $\mu$-a.e.\ $x\in\R^d$, notice that there exists 
a countable union of possibly rotated $n$-dimensional Lipschitz graphs $\Gamma_i$ such that $\mu$
is absolutely continuous with respect to $\HH^n|_{\bigcup_i\Gamma_i}$ (indeed, in the definition \rf{eq001} of $n$-rectifiable sets, one can replace the sets $f_i(\R^n)$ by 
possibly rotated $n$-dimensional  Lipscthiz graphs $\Gamma_i$).
 Thus, it suffices to show that
$T\mu(x)<\infty$ for $\HH^n$-a.e.\ $x\in\Gamma_i$. This is an immediate
consequence of the estimate \rf{eqdeb0} applied to the particular case $\Gamma=\Gamma_i$, $\nu=\mu$, recalling that we assume $\|\mu\|<\infty$.

The first step to prove the key estimate \rf{eqdeb0} consists in showing that
$T_\mu$ is bounded in $L^2(\mu)$ if $\mu$ is a uniformly $n$-rectifiable measure. This is shown in the next section. Later, by means of a suitable Calder\'on-Zygmund decomposition we will prove that 
$T$ is bounded from $M(\R^d)$ to $L^{1,\infty}(\mu)$, which in particular yields the estimate
\rf{eqdeb0} just by choosing $\mu=\HH^n|_\Gamma$.


\section{$T_\mu$ is bounded in $L^2(\mu)$ if $\mu$ is uniformly rectifiable}

The objective of this section consists in proving the following.

\begin{theorem}\label{teol2}
Let $\mu$ be a uniformly $n$-rectifiable measure in $\R^d$.
Then $T_\mu$ is bounded in $L^2(\mu)$.
\end{theorem}

The following theorem, which is one main results from \cite{CGLT},  will be a fundamental ingredient of the proof of Theorem \ref{teol2}.

\begin{theorem}\label{teocglt}
Let $\mu$ be a uniformly $n$-rectifiable measure in $\R^d$. Then
there exists a constant $c$ such that for any ball $B\subset\R^d$ with radius $R$,
$$\int_0^R\int_{x\in B} \left|\frac{\mu(B(x,r))}{r^n} - \frac{\mu(B(x,2r))}{(2r)^n}\right|^2\,d\mu(x)\,\frac{dr}r\leq c\,R^n,$$
with $c$ depending only on $n,d$ and the constants involved in the AD-regular and
uniformly $n$-rectifiable character of $\mu$.
\end{theorem}

We are ready now to prove Theorem \ref{teol2}.
The proof below is somewhat similar in spirit to some of the arguments in \cite{Mas-Tolsa}.

\vv
\begin{proof}[\bf Proof of Theorem \ref{teol2}]
For $k\in\Z$, $f\in L^2(\mu)$, $x\in\R^d$, we denote
$$T_{\mu,k} f(x) = \left(\int_{2^{-k-2}}^{2^{-k-1}} \left|\frac{(f\mu)(B(x,r))}{r^n} - \frac{
(f\mu)(B(x,2r))}{(2r)^n}\right|^2\,\frac{dr}r\right)^{1/2},$$
so that
$$T_\mu f(x)^2 = \sum_{k\in\Z}T_{\mu,k} f(x)^2.$$

Consider the family $\DD$ of dyadic cubes associated with $\mu$, as described in Section \ref{secprelim}, and write
$$\int |T_\mu f|^2\,d\mu = \sum_{k\in\Z}\int |T_{\mu,k} f|^2\,d\mu = \sum_{Q\in\DD} 
\int |T_Q f|^2\,d\mu,$$
where
$$T_Q f = \chi_Q T_{\mu,J(Q)}f.$$
Recall that $J(Q)$ stands for integer $k$ such that $Q\in\DD_k$.
Observe that if $x\in Q$ and $r\leq \ell(Q)/2$, then 
$$B(x,2r)\cap\supp(\mu)\subset \bigcup_{P\in\NN(Q)} P.$$
So denoting $\wt Q = \bigcup_{P\in\NN(Q)} P$ we deduce that
$$T_Q f = T_Q(\chi_{\wt Q} f).$$

By the martingale decomposition \rf{eqmartingale}, we have
\begin{align*}
\chi_{\wt Q} f & = \sum_{R\in\NN(Q)} \Bigl(\chi_R \, m_R f + \sum_{P\in\DD(R)}\Delta_P f\Bigr)\\
&= \chi_{\wt Q}\,m_Qf +  \sum_{R\in\NN(Q)} \chi_R \, (m_R f - m_Q f) +
\sum_{R\in\NN(Q)}  \sum_{P\in\DD(R)}\Delta_P f.
\end{align*}
Therefore,
\begin{align}\label{eqw10}
T_Q f &\leq |m_Q f|\, T_Q\chi_{\wt Q} + \sum_{R\in\NN(Q)} |m_R f - m_Qf| \,T_Q\chi_R+
\sum_{R\in\NN(Q)}T_Q\biggl(\sum_{P\in\DD(R)}\Delta_P f\biggr) \\
&=: A_Q f + B_Q f + C_Q f.\nonumber
\end{align}
So we have
\begin{equation}\label{eqw11}
\int |T_\mu f|^2\,d\mu \lesssim \sum_{Q\in\DD} 
\int |A_Q f|^2\,d\mu + \sum_{Q\in\DD}\int |B_Q f|^2\,d\mu + \sum_{Q\in\DD}\int |C_Q f|^2\,d\mu.
\end{equation}

To estimate the first sum on the right side of the preceding inequality notice that for any
$S\in\DD$
\begin{align*}
\sum_{Q\in \DD(S)} \int|T_Q \chi_{\wt Q}|^2\,d\mu &= \sum_{Q\in \DD(S)} \int|T_Q 1|^2\,d\mu\\
&= 
\int_0^{\ell(S)/2}\!\!\int_{x\in S} \left|\frac{\mu(B(x,r))}{r^n} - \frac{\mu(B(x,2r))}{(2r)^n}\right|^2\,d\mu(x)\,\frac{dr}r\leq c\,\mu(S),
\end{align*}
by Theorem \ref{teocglt}. Then, by the Carleson embedding theorem, we deduce that
$$\sum_{Q\in\DD} 
\int |A_Q f|^2\,d\mu = \sum_{Q\in\DD} |m_Qf|^2
\int |T_Q \chi_{\wt Q}|^2\,d\mu \leq c\,\|f\|_{L^2(\mu)}^2.$$

To estimate the second sum on the right side of \rf{eqw11} we denote
$$b_Q(f) = \sum_{R\in\NN(Q)}|m_Rf - m_Qf|,$$
and then we write
$$ B_Q f =
\sum_{R\in\NN(Q)} |m_R f - m_Qf| \,T_Q\chi_R \leq c\,b_Q(f)\,\chi_Q,$$
just using the trivial estimate $T_Q\chi_R\lesssim\chi_Q$. Thus
$$\sum_{Q\in\DD}\int |B_Q f|^2\,d\mu \leq c 
\sum_{Q\in\DD} b_Q(f)^2\,\mu(Q).$$
As shown in Proposition 5.9 of \cite{Mas-Tolsa}, the last sum is bounded by 
$c\,\|f\|_{L^2(\mu)}^2$. 

So it only remains to show that
\begin{equation}\label{eqww00}
\sum_{Q\in\DD}\int |C_Q f|^2\,d\mu\leq c\,\|f\|_{L^2(\mu)}^2.
\end{equation}
To this end we set
\begin{equation}\label{eqww01}
|C_Q f|^2 \lesssim \sum_{R\in\NN(Q)}\biggl|T_Q\biggl(\sum_{P\in\DD(R)}\Delta_P f\biggr)\biggr|^2,
\end{equation}
taking into account that number of neighbors of $Q$ is uniformly bounded.

For $x\in Q$ and $R\in\NN(Q)$ we have
\begin{multline}\label{eqww3}
T_Q\biggl(\sum_{P\in\DD(R)}\Delta_P f\biggr)(x)^2 \\
= \int_{\ell(Q)/4}^{\ell(Q)/2}
\left|\frac{\sum_{P\in\DD(R)}(\Delta_P f\,\mu)(B(x,r))}{r^n} - \frac{
\sum_{P\in\DD(R)}(\Delta_P f\,\mu)(B(x,2r))}{(2r)^n}\right|^2\,\frac{dr}r.
\end{multline}
Recall now that $\int \Delta_P f\,d\mu =0$, and so
$(\Delta_P f\,\mu)(B(x,r))=0$ unless both $P \cap B(x,r)\neq\varnothing$ and $P \cap B(x,r)^c\neq\varnothing$, and analogously replacing $r$ by $2r$. So if we denote
\begin{align*}
J_{R,r}(x) & = \bigl\{P\in\DD(R): P \cap B(x,r)\neq\varnothing\mbox{ and } P \cap B(x,r)^c\neq\varnothing\bigr\} \\
& \quad\cup \bigl\{P\in\DD(R): P \cap B(x,2r)\neq\varnothing\mbox{ and } P \cap B(x,2r)^c\neq\varnothing\bigr\},
\end{align*}
then we have
$$
\left|\frac{\sum_{P\in\DD(R)}(\Delta_P f\,\mu)(B(x,r))}{r^n} - \frac{
\sum_{P\in\DD(R)}(\Delta_P f\,\mu)(B(x,2r))}{(2r)^n}\right| \leq 
\frac c{\ell(Q)^n}\sum_{P\in J_{R,r}(x)} \|\Delta_P f\|_{L^1(\mu)}.$$
By Cauchy-Schwarz applied twice, the right hand side is bounded by
\begin{multline*}
\frac c{\ell(Q)^n}\Biggl(\sum_{P\in J_{R,r}(x)} \ell(P)^{n-1/2}\Biggr)^{1/2}
\Biggl(\sum_{P\in J_{R,r}(x)}\frac1{\ell(P)^{n-1/2}}\,
\|\Delta_P f\|_{L^1(\mu)}^2\Biggr)^{1/2}\\
\leq\frac c{\ell(Q)^n}\Biggl(\sum_{P\in J_{R,r}(x)} \ell(P)^{n-1/2}\Biggr)^{1/2}
\Biggl(\sum_{P\in \DD(R)}\ell(P)^{1/2}\,
\|\Delta_P f\|_{L^2(\mu)}^2\Biggr)^{1/2}.
\end{multline*}
Plugging this estimate into \rf{eqww3} we obtain
\begin{multline}  \label{eqww55}
T_Q\biggl(\sum_{P\in\DD(R)}\Delta_P f\biggr)(x)^2 \\
\leq
\frac c{\ell(Q)^{2n}}
\Biggl(\sum_{P\in \DD(R)}\ell(P)^{1/2}\,
\|\Delta_P f\|_{L^2(\mu)}^2\Biggr)
\int_{\ell(Q)/4}^{\ell(Q)/2} \sum_{P\in J_{R,r}(x)} \ell(P)^{n-1/2}\,\frac{dr}r.
\end{multline}
To estimate the last integral, notice if $P\in J_{R,r}(x)$, then either $\partial B(x,r)$ 
or $\partial B(x,2r)$ intersect the convex hull of $P$, which we denote by ${\rm conv}(P)$.
By Fubini then we get
\begin{multline*}
\int_{\ell(Q)/4}^{\ell(Q)/2} \sum_{P\in J_{R,r}(x)} \ell(P)^{n-1/2}\,\frac{dr}r\\
 \leq \frac c{\ell(Q)} \sum_{P\in\DD(R)} \ell(P)^{n-1/2}\,\,
\bigl|\bigl\{r>0: \bigl[(\partial B(x,r))\cup (\partial B(x,2r))\bigr]\cap {\rm conv}(P)\neq\varnothing\bigr\}\bigr|.\end{multline*}
Since $\diam(P)\approx\ell(P)$, for any fixed $x$ we have
\begin{equation}\label{tt1}
\bigl|\bigl\{r>0: (\partial B(x,r))\cap {\rm conv}(P)\neq\varnothing\bigr\}\bigr| \lesssim \ell(P),
\end{equation}
and analogously replacing $\partial B(x,r)$ by $\partial B(x,2r)$. So we obtain
$$\int_{\ell(Q)/4}^{\ell(Q)/2} \sum_{P\in J_{R,r}(x)} \ell(P)^{n-1/2}\,\frac{dr}r
\leq \frac c{\ell(Q)} \sum_{P\in\DD(R)} \ell(P)^{n+1/2} \leq c\,\frac {\ell(Q)^{n+1/2}}{\ell(Q)}
= c\,\ell(Q)^{n-1/2}.$$
Together with \rf{eqww55} this gives us that, for all $x\in\R^d$,
$$T_Q\biggl(\sum_{P\in\DD(R)}\Delta_P f\biggr)(x)^2
\leq
\frac c{\ell(Q)^{n+1/2}}
\sum_{P\in \DD(R)}\ell(P)^{1/2}\,
\|\Delta_P f\|_{L^2(\mu)}^2\,\chi_Q(x).
$$

From the last estimate and \rf{eqww01} we deduce that
\begin{align*}
\sum_{Q\in\DD}\int |C_Q f|^2\,d\mu & \leq c\sum_{Q\in\DD} \sum_{R\in\NN(Q)}\int
\biggl|T_Q\biggl(\sum_{P\in\DD(R)}\Delta_P f\biggr)\biggr|^2\,d\mu\\
& \leq c\sum_{Q\in\DD} \sum_{R\in\NN(Q)}\frac 1{\ell(Q)^{1/2}}
\sum_{P\in \DD(R)}\ell(P)^{1/2}\,
\|\Delta_P f\|_{L^2(\mu)}^2\\
& = c\sum_{P\in\DD} \|\Delta_P f\|_{L^2(\mu)}^2 \sum_{Q\in\DD}\sum_{R\in\NN(Q):R\supset P}\frac {\ell(P)^{1/2}}{\ell(Q)^{1/2}}.
\end{align*}
Since
$$\sum_{Q\in\DD}\sum_{R\in\NN(Q):R\supset P}\frac {\ell(P)^{1/2}}{\ell(Q)^{1/2}}\lesssim 1,$$
\rf{eqww00} follows, and the proof of the theorem is concluded.
\end{proof}


\section{$T$ is bounded from $M(\R^d)$ to $L^{1,\infty}(\mu)$ if $\mu$ is uniformly rectifiable}

In this section we will prove the following.

\begin{theorem}\label{teol1}
Let $\mu$ be a uniformly $n$-rectifiable measure in $\R^d$.
Then $T$ is bounded from $M(\R^d)$ to $L^{1,\infty}(\mu)$. That is, there exists some constant
$c$ such that 
\begin{equation}\label{eqdebbb}
\mu\bigl(\{x\in\R^d: \,T\nu(x)>\lambda\}\bigr) \leq c\,\frac{\|\nu\|}\lambda\quad
\mbox{\;for
all $\nu\in M(\R^d)$ and  $\lambda>0$.}
\end{equation}
\end{theorem}

As shown in Section \ref{secqu}, this result implies that if
$\mu$ is an $n$-rectifiable measure in $\R^d$, then
$$\int_0^1 \left|\frac{\mu(B(x,r))}{r^n} - \frac{\mu(B(x,2r))}{(2r)^n}\right|^2\,\frac{dr}r< \infty
\quad
\mbox{for $\mu$-a.e.\ $x\in\R^d$.}
$$

Before proving Theorem \ref{teol1} we state the Calder\'on-Zygmund decomposition we need.

\begin{lemma} \label{lemcz}
Let $\mu$ be an $n$-AD-regular measure. For every $\nu\in M(\R^d)$ with compact support and every $\lambda>2^{d+1}\|\nu\|/\|\mu\|$ we have:
\begin{itemize}
\item[(a)] There exists a finite or countable collection of cubes $\{Q_j\}_{j\in J}$ with bounded overlap
(that is, $\sum_j\chi_{Q_j}\leq c$) and a function $f\in L^1(\mu)$ such that, for each $j\in J$,
\begin{equation}\label{eqaa1}
|\nu|(Q_j)> 2^{-d-1}\lambda\,\mu(2Q_j),
\end{equation}
\begin{equation}\label{eqaa2}
|\nu|(\eta Q_j)\leq 2^{-d-1}\lambda\,\mu(2\eta Q_j)\quad \mbox{ for every $\eta>2$,}
\end{equation}
and moreover,
\begin{equation}\label{eqaa3}
\nu=f\mu \quad\mbox{in $\R^d\setminus \bigcup_{j\in J} Q_j,\;$ with $|f|\leq\lambda$ $\mu$-a.e.}
\end{equation}

\item[(b)] For each $j\in J$, let $R_j=6Q_j$ and denote $w_j=\chi_{Q_j}\left(\sum_k\chi_{Q_k}\right)^{-1}$.
There exists a family of functions $\{b_j\}_{j\in J}$ with $\supp b_j\subset R_j$, each one with constant
sign, such that
\begin{equation}\label{eqaa4}
\int b_j\,d\mu = \int w_j\,d\nu,
\end{equation}
\begin{equation}\label{eqaa5}
\|b_j\|_{L^\infty(\mu)}\,\mu(R_j) \leq c\,|\nu|(Q_j),
\end{equation}
\begin{equation}\label{eqaa6}
\sum_{j\in J}|b_j|\leq c\,\lambda.
\end{equation}
\end{itemize}
\end{lemma}

Let us remark that the cubes in the preceding lemma are not cubes from $\DD$, but true cubes.
Abusing notation, the side length of such a cube $Q$ will be denoted also by $\ell(Q)$.
Observe also that, in particular, that \rf{eqaa2} implies that $4\bar Q_j\cap\supp(\mu)\neq\varnothing$ and
thus $\mu(R_j)\approx\ell(R_j)^n$.
 
For the proof of the lemma the reader can see Lemma 2.14 of \cite{Tolsa-llibre}, where this is
shown in the more general situation where $\mu$ need not be doubling.

\vv

\begin{proof}[\bf Proof of Theorem \ref{teol1}]
Suppose first that $\nu\in M(\R^d)$ has compact support. 
 Clearly, we may assume that $\lambda>2^{d+1}\|\nu\|/\|\mu\|$.

For such $\lambda>0$, consider $Q_j,R_j,w_j,b_j$, for $j\in J$, and $f$ as
 in Lemma \ref{lemcz}. Then write $\nu= g\,\mu + \beta$, where
$$g\,\mu =\chi_{\R^d\setminus \bigcup_{j\in J} Q_j}\,\nu + \sum_{j\in J} b_j\,\mu$$
and
$$\beta=\sum_{j\in J}\beta_j := \sum_{j\in J}(w_j\,\nu - b_j\,\mu).$$
Observe that $\|g\|_{L^\infty(\mu)}\leq c\,\lambda$ and, for each $j\in J$, 
$$\supp(\beta_j)\subset R_j\quad \mbox{ and }\quad \beta_j(R_j)=0.$$
So $\beta_j$ is a real measure with zero mean.

By \rf{eqaa1} we have
$$\mu\Bigl(\bigcup_j 2Q_j\Bigr) \leq \frac c\lambda\, \sum_j |\nu|(Q_j) \leq \frac c\lambda
\,\|\nu\|.$$
So we only have to check that
\begin{equation}\label{eqaa7}
\mu\Bigl(\Bigl\{x\in\R^d\setminus \bigcup_j 2Q_j: T\nu(x)>\lambda\Bigr\}\Bigr)
\leq \frac c\lambda
\,\|\nu\|.
\end{equation}
Taking into account that $T_\mu$ is bounded in $L^2(\mu)$, using the fact that $\|g\|_{L^\infty(\mu)}\le c\lambda$ we derive
$$\mu\Bigl(\Bigl\{x\in\R^d\setminus \bigcup_j 2Q_j: T_\mu g(x)>\lambda/2\Bigr\}\Bigr)
\leq  \frac c{\lambda^2} \int |g|^2\,d\mu \leq \frac c{\lambda} \int |g|\,d\mu.$$
Also, by the definition of $g$ and \rf{eqaa5} we get
$$\int |g|\,d\mu \leq |\nu|
\Bigl(\R^d\setminus \bigcup_{j\in J} Q_j\Bigr) + \sum_{j\in J} \int |b_j|\,d\mu
\leq \|\nu\| + c\sum_{j\in J} |\nu|(Q_j) \leq c\,\|\nu\|.$$
Thus
\begin{equation}\label{eqgg0}
\mu\Bigl(\Bigl\{x\in\R^d\setminus \bigcup_j 2Q_j: T_\mu g(x)>\lambda/2\Bigr\}\Bigr)
\leq \frac c\lambda\,\|\nu\|.
\end{equation}

Let us turn attention to $T\beta$ now. We set
\begin{align}\label{eqgg1}
\mu\Bigl(\Bigl\{x\in\R^d\setminus \bigcup_j 2Q_j: T \beta(x)>\lambda/2\Bigr\}\Bigr)
& \leq \frac2\lambda\int_{\R^d\setminus \bigcup_j 2Q_j} T\beta\,d\mu \\
& \leq \frac2\lambda\,\sum_{j\in J}\int_{\R^d\setminus 2R_j} T\beta_j\,d\mu +  
\frac2\lambda\,\sum_{j\in J}\int_{2R_j\setminus 2Q_j} T\beta_j\,d\mu.\nonumber
\end{align}
First we will estimate the first sum on the right side. To this end, notice that
 since
$\beta_j$ has zero mean and is supported on $R_j$, for each $x\in\R^d\setminus 2R_j$ and $r>0$ we deduce that $\beta_j(B(x,r))=0$ unless
$\partial B(x,r)\cap R_j\neq\varnothing$. Of course, the analogous statement holds for $\beta_j(B(x,2r))$. Thus
\begin{align*}
|T\beta_j(x)|^2 &\leq \int_0^\infty \left(\frac{|\beta_j(B(x,r))|}{r^n} + \frac{|\beta_j(B(x,2r))|}{(2r)^n}\right)^2\,\frac
{dr}r\\
&\leq 2\int_{\{r:\partial B(x,r)\cap R_j\neq\varnothing\}} \left(\frac{|\beta_j|(R_j)}{r^n}\right)^2\,\frac
{dr}r + 2\int_{\{r:\partial B(x,2r)\cap R_j\neq\varnothing\}} \left(\frac{|\beta_j|(R_j)}{(2r)^n}\right)^2\,\frac
{dr}r.
\end{align*}
Observe that if $\partial B(x,r)\cap R_j\neq\varnothing$ or $\partial B(x,2r)\cap R_j\neq\varnothing$, then $r\approx |x-x_j|$ where $x_j$ stands for the center of $R_j$ (and of $Q_j$), because
$x\not\in 2R_j$. By (\ref{tt1}) we have 
$$|\{r>0:\partial B(x,r)\cap R_j\neq\varnothing\}|\leq \diam(R_j)\leq c\,\ell(R_j),$$
and analogously replacing $r$ by $2r$. Therefore we obtain
$$
|T\beta_j(x)|^2 \leq c\,\frac{|\beta_j|(R_j)^2}{|x-x_j|^{2n+1}}
\biggl(\int_{\{r:\partial B(x,r)\cap R_j\neq\varnothing\}}dr + \int_{\{r:\partial B(x,2r)\cap R_j\neq\varnothing\}}dr\biggr) \leq c\,\frac{|\beta_j|(R_j)^2\,\ell(R_j)}{|x-x_j|^{2n+1}}.
$$
From this estimate, using also the upper $n$-AD-regularity
of $\mu$ we infer that
\begin{equation}\label{eqgg7.5}
\int_{\R^d\setminus 2R_j} |T\beta_j|\,d\mu
\leq 
c \,|\beta_j|(R_j)\,\ell(R_j)^{1/2}\int_{\R^d\setminus 2R_j} \frac{1}{|x-x_j|^{n+1/2}}\,d\mu(x)
\leq c \,|\beta_j|(R_j)\leq c \,|\nu|(Q_j).
\end{equation}

To estimate the last term in \rf{eqgg1} we set
\begin{align*}
\int_{2R_j\setminus 2Q_j} T\beta_j\,d\mu&\leq 
c\,\ell(R_j)^{n/2}\left(\int_{2R_j\setminus 2Q_j} |T\beta_j|^2\,d\mu\right)^{1/2}\\
& \leq c\,\ell(R_j)^{n/2} \left[\left(\int_{2R_j\setminus 2Q_j} |T(w_j\nu)|^2\,d\mu\right)^{1/2}
+  \left(\int |T(b_j\mu)|^2\,d\mu\right)^{1/2}\right]
.
\end{align*}
By the $L^2(\mu)$ boundedness of $T_\mu$ and the condition \rf{eqaa5} we have
\begin{equation}\label{eqgg8}
\ell(R_j)^{n/2}\left(\int |T(b_j\mu)|^2\,d\mu\right)^{1/2}\leq c\,\ell(Q_j)^{n/2}\|b_j\|_{L^2(\mu)}\leq c\,\ell(Q_j)^{n}\,\|b_j\|_{L^\infty(\mu)}\leq c\,|\nu|(Q_j).
\end{equation}
On the other hand, for $x\in 2R_j\setminus 2Q_j$, since $\dist(x,\supp(\nu|_{Q_j}))\geq \ell(Q_j)/2$, we have
$$|T(w_j\nu)(x)|^2\leq c\,\int_{r\geq \ell(Q_j)/4} \frac{|\nu|(Q_j)^2}{r^{2n}}\,\frac{dr}r
\leq c\,\frac{|\nu|(Q_j)^2}{\ell(Q_j)^{2n}}.$$
Therefore,
$$\ell(R_j)^{n/2} \left(\int_{2R_j\setminus 2Q_j} |T(w_j\nu)|^2\,d\mu\right)^{1/2}
\leq c\,\ell(R_j)^{n/2}\,\frac{|\nu|(Q_j)}{\ell(Q_j)^n} \,\mu(2R_j)^{1/2} \leq c\,|\nu|(Q_j).$$
From this estimate and \rf{eqgg8} we deduce that
$$\int_{2R_j\setminus 2Q_j} T\beta_j\,d\mu \leq c\,|\nu|(Q_j).$$
Together with \rf{eqgg1} and \rf{eqgg7.5} this yields
$$\mu\Bigl(\Bigl\{x\in\R^d\setminus \bigcup_j 2Q_j: T \beta(x)>\lambda/2\Bigr\}\Bigr)
\leq \frac c\lambda\,\sum_{j\in J}|\nu|(Q_j)\leq \frac c\lambda\,\|\nu\|,$$
which finishes the proof of \rf{eqaa7}.

\vspace{2mm}
Suppose now that $\nu$ is not compactly supported. 
Let $N_0,N_1$ be two big positive integers, with $N_1\geq 2 N_0$.
Denote 
$\nu_{N_1} = \chi_{B(0,N_1)}\,\nu$.
It is easy to check that for $x\in B(0,N_0)$
$$T(\chi_{\R^d\setminus B(0,N_1)}\nu)(x)\leq
c\,\frac{|\nu|(\R^d\setminus B(0,N_1))}{N_1-N_0}
\leq c\,\frac{\|\nu\|}{N_1-N_0}\leq \frac\lambda2,$$
assuming $N_1$ big enough so that
$N_1-N_0> 2c\|\nu\|/\lambda$.
Thus, for such $N_1$, since $\nu_{N_1}$ has compact support, 
\begin{align*}
\mu\bigl(\{x\in B(0,N_0):|T\nu(x)|>\lambda\}\bigr)& \leq \mu\bigl(\{x\in B(0,N_0):|T\nu_{N_1}(x)|>\lambda/2\}
\bigr)\\&\leq c\,\frac{\|\nu_{N_1}\|}\lambda\leq c\,\frac{\|\nu\|}\lambda.
\end{align*}
Since $N_0$ is arbitrary and this estimate is uniform on $N_0$, \rf{eqdebbb} follows in full generality.
\end{proof}


\section{Finiteness of the square function implies $n$-rectifiability}\label{sec6}

In this section we will prove the implication (b)$\Rightarrow$(a) from Theorem \ref{teomain}.
We have to show that
if $\mu$ is a Radon measure on $\R^d$ such that, for $\mu$-a.e.\ $x\in\R^d$, $0< \Theta_*^n(x,\mu)\leq \Theta^{n,*}(x,\mu)<\infty$ and
\begin{equation}\label{eq110}
\int_0^1 \left|\frac{\mu(B(x,r))}{r^n} - \frac{\mu(B(x,2r))}{(2r)^n}\right|^2\,\frac{dr}r< \infty
,
\end{equation}
then $\mu$ is $n$-rectifiable.  Without loss of generality, to prove this result we can assume $\mu$ to be compactly
supported, and thus finite.

As in \cite{CGLT}, we denote
$$\Delta_\mu(x,r) := \frac{\mu(B(x,r))}{r^n} - \frac{\mu(B(x,2r))}{(2r)^n}.$$
Notice that for any $r>0$, $|\Delta_\mu(x,r)| \leq \|\mu\|/r^n$. Thus,
$$\int_1^\infty \Delta_\mu(x,r)^2\,\frac{dr}r \leq \|\mu\|^2 \int_1^\infty \frac{dr}{r^{2n+1}} \leq c\,\|\mu\|^2.$$
Therefore, \rf{eq110} is equivalent to 
\begin{equation}\label{eq111}
\int_0^\infty \Delta_\mu(x,r)^2\,\frac{dr}r <\infty.
\end{equation}

We consider the auxiliary function
$\vphi(x) = e^{-|x|^2}$ and, for $r>0$, we set
$$\vphi_r(x) = \frac1{r^n}\,\vphi\left(\frac xr\right).$$
Then we define 
$$\Delta_{\mu,\vphi}(x,r) := \vphi_r * \mu - \vphi_{2r}*\mu(x) = 
\int \bigl(\vphi_r(x-y) - \vphi_{2r}(x-y)\bigr)\,d\mu(y).$$

Arguing as in the proof of Corollary 3.12 from \cite{CGLT} it is easy to check that if
$\sup_{s>0}\frac{\mu(B(x,s))}{s^n}<\infty$, then
 $\Delta_{\mu,\vphi}(x,r)$ can be written
as a suitable convex combination of $\Delta_{\mu}(x,s)$, $s>0$.
 We show the details for completeness. First we look for a function $\wt \vphi_r:(0,\infty) \to (0,\infty)$ such that
\begin{equation}
\label{convfi}
\frac{1}{r^n} e^{\frac{-t^2}{r^2}}=\int_0^\infty \frac1{s^n}\chi_{[0,s]}(t) \,\wt \vphi_r(s) \,ds=\int_t^\infty \frac{\wfi(s)}{s^n}\,ds,\quad \mbox{ for } t>0.
\end{equation}
Differentiating with respect to $t$ we obtain
$$-\frac{2 t}{r^{n+2}}\,e^{\frac{-t^2}{r^2}}=- \frac{\wfi(t)}{s^n}.$$
Thus, for $r>0$ and  $t>0$, 
$$\wfi(t)=\frac{2 t^{n+1}}{r^{n+2}}e^{\frac{- t^2}{r^2}}.$$

Using \rf{convfi} we can now write
\begin{align*}
\Delta_{\mu,\vphi}(x,r)& =(\vphi_r -\vphi_{2r}) * \mu(x)\\
& = \left( \int_0^\infty \!\frac 1{s^n} \chi_{[0,s]}(|\cdot |) \,\wfi (s)\,ds)\right) \ast \mu (x) -\left( \int_0^\infty \frac 1{s^n} \chi_{[0,s]}(|\cdot |) \,\wf_{2r} (s)\,ds\right) * \mu (x)\end{align*}
By a change of  variables we have
$$\int_0^\infty \frac 1{s^n} \chi_{[0,s]}(|y-x|) \,\wf_{2r} (s)\,ds= \int_0^\infty \!\frac 1{(2s)^n} \chi_{[0,2s]}(|y-x|) \,\wfi(s)\, ds.$$
Therefore,
\begin{align*}
\Delta_{\mu,\vphi}(x,r)&= \int_0^{\infty} \left( \frac1{s^n} \chi_{B(0,s)}(\cdot)-\frac1{(2s)^n} \chi_{B(0,2s)}(\cdot)\right) * \mu (x) \, \wfi(s)\, ds = 
\int_0^{\infty} 
\Delta_\mu(x,s) \, \wfi(s)\, ds,
\end{align*}
as wished.

By Cauchy-Schwarz, then we obtain
$$\Delta_{\mu,\vphi}(x,r)^2\leq \left(\int_0^{\infty} 
\Delta_\mu(x,s)^2 \, \wfi(s)\, ds\right) 
\left(\int_0^{\infty} 
 \wfi(s)\, ds\right) \leq c\,\int_0^{\infty} 
\Delta_\mu(x,s)^2 \, \wfi(s)\, ds.$$
Thus,
$$\int_0^\infty \Delta_{\mu,\vphi}(x,r)^2\,\frac{dr}r \leq c\,
\int_0^\infty \int_0^{\infty} 
\Delta_\mu(x,s)^2 \, \wfi(s)\, ds\,\frac{dr}r.$$
Since
\begin{equation*}
\int_0^\infty \wfi (s) \, \frac{dr}{r}=2 \int_0^\infty \left(\frac{s}{r}\right)^{n+1} e^{\frac{-  s^2}{r^2}}\, \frac{dr}{r^2}=\frac2s  \int_0^\infty  t^{n+1}e^{- t^2}\,dt \lesssim \frac1s
\end{equation*}
we get
$$\int_0^\infty \Delta_{\mu,\vphi}(x,r)^2\,\frac{dr}r \leq c\,
\int_0^\infty 
\Delta_\mu(x,s)^2\, \frac{ds}s.$$

So \rf{eq111} implies that
\begin{equation}\label{eq112}
\int_0^\infty \Delta_{\mu,\vphi}(x,r)^2\,\frac{dr}r <\infty.
\end{equation}

\begin{lemma}\label{lemlimit}
Let $\mu$ be a finite Borel measure in $\R^d$ and $x\in\R^d$ such that $ \Theta^{n,*}(x,\mu)<\infty$. If
$$\int_0^\infty \Delta_{\mu,\vphi}(x,r)^2\,\frac{dr}r <\infty,$$
then 
$$\lim_{r\to0} \Delta_{\mu,\vphi}(x,r)=0.$$
\end{lemma}

\begin{proof}
Denote $I_k=[2^{-k},\,2^{-k+1}]$ and 
$$\lambda_k= \frac1{|I_k|}\int_{I_k} \Delta_{\mu,\vphi}(x,r)^2\,dr.$$
Since
$\sum_{k\in \Z} \lambda_k<\infty$,
we have
$$\lim_{k\to\infty} \lambda_k=0.$$
By Chebichev, we get
$$\bigl|\{r\in I_k:\,|\Delta_{\mu,\vphi}(x,r)|>\lambda_k^{1/3}\}\bigr|
\leq \frac1{\lambda_k^{2/3}} \int_{I_k} \Delta_{\mu,\vphi}(x,r)^2\,dr\leq 
\frac{\lambda_k\,|I_k|}{\lambda_k^{2/3}} = \lambda_k^{1/3}\,|I_k|.$$
Thus, assuming $\lambda_k<1$, for any $r\in I_k$ there exists some $r'\in I_k$ satisfying
\begin{equation}\label{eqaa120}
|r-r'|\leq \lambda_k^{1/3}\,|I_k|\quad\mbox{ and } \quad|\Delta_{\mu,\vphi}(x,r')|\leq\lambda_k^{1/3}.
\end{equation}

Now we wish to estimate the difference between $\Delta_{\mu,\vphi}(x,r)$ and $\Delta_{\mu,\vphi}(x,r')$, for $r,r'\in I_k$. By the mean value theorem, we have
$$\bigl|\Delta_{\mu,\vphi}(x,r)-\Delta_{\mu,\vphi}(x,r')\bigr|\leq |r-r'|\,
\sup_{s\in I_k}|\partial_s \Delta_{\mu,\vphi}(x,s)|.$$
Notice that
$$|\partial_s \Delta_{\mu,\vphi}(x,s)| \leq \bigl|\partial_s (\vphi_s*\mu(x)) \bigr|
+\bigl|\partial_s (\vphi_{2s}*\mu(x)) \bigr|.
$$
We have 
\begin{align}\label{eq119}
\bigl|\partial_s (\vphi_s*\mu(x)) \bigr| & = \left|\int \partial_s \left(\frac1{s^n}\,e^{-|x-y|^2/s^2}\right)\,
d\mu(y)\right|\\
& \leq
 \int \left(\frac n{s^{n+1}}+ \frac 2{s^{n+1}}\, \frac{|x-y|^2}{s^2}\right)
 e^{-|x-y|^2/s^2}\,d\mu(y)\nonumber\\
 & = \frac1s 
 \int \frac1{s^n}\left( n +  \frac{2\,|x-y|^2}{s^2}\right)
 e^{-|x-y|^2/s^2}\,d\mu(y).\nonumber
\end{align}
Denote
$$b_x=\sup_{r>0}\frac{\mu(B(x,r))}{r^n}.$$
Observe that $b_x<\infty$ because $ \Theta^{n,*}(x,\mu)<\infty$ and $\|\mu\|<\infty$.
Using the fast decay of the function inside the integral on the right side of \rf{eq119} and
splitting the domain of integration into annuli, it
follows easily that
$$\int \frac1{s^n}\left( n +  \frac{2\,|x-y|^2}{s^2}\right)
 e^{-|x-y|^2/s^2}\,d\mu(y)\lesssim b_x.$$
So for $s\in I_k$ we obtain
$$\bigl|\partial_s (\vphi_s*\mu(x))\bigr| \lesssim \frac{b_x}s \approx \frac{b_x}{|I_k|}.$$
An analogous estimate holds replacing $s$ by $2s$, and thus
$|\partial_s \Delta_{\mu,\vphi}(x,s)| \lesssim b_x\,|I_k|^{-1}$.
Then for $r,r'\in I_k$ we get 
$$\bigl|\Delta_{\mu,\vphi}(x,r)- \Delta_{\mu,\vphi}(x,r')\bigr|\lesssim\frac{b_x}{|I_k|}\,|r-r'|.$$

Let $k$ be big enough so that $\lambda_k<1$. 
Given any $r\in I_k$ we can take $r'\in I_k$ satisfying \rf{eqaa120}.
From the last estimate we deduce that
$$\bigl|\Delta_{\mu,\vphi}(x,r)\bigr|\leq \bigl|\Delta_{\mu,\vphi}(x,r')\bigr| + \bigl|\Delta_{\mu,\vphi}(x,r)- \Delta_{\mu,\vphi}(x,r')\bigr| \lesssim \lambda_k^{1/3} + \frac{b_x}{|I_k|}\,|r-r'|
\leq (1+b_x)\lambda_k^{1/3},$$
which tends to $0$ as $k\to\infty$. Thus, the lemma follows.
\end{proof}

Our next objective consists in proving the following.

\begin{propo}\label{propoclau}
Let $\mu$ be a finite Borel measure in $\R^d$ such that, for $\mu$-a.e.\ $x\in\R^d$, $0< \Theta_*^n(x,\mu)\leq \Theta^{n,*}(x,\mu)<\infty$ and moreover
$$\lim_{r\to0} \Delta_{\mu,\vphi}(x,r)=0.$$
Then $\mu$ is $n$-rectifiable.
\end{propo}

It is clear that the preceding proposition together with Lemma \ref{lemlimit} completes the proof of (b)$\Rightarrow$(a) from Theorem \ref{teomain}. 

We will prove Proposition \ref{propoclau} by using the so called tangent measures.
Given $x\in\R^d$ and $r>0$, denote by $T_{x,r}$ the homothety that
maps $B(x,r)$ to $B(0,1)$. That is,
$$T_{x,r}(y) = \frac1r\,(y-x).$$
Observe that the image measure of $\mu$ by $T_{x,r}$ satisfies 
$$T_{x,r}\#\mu(A) =\mu(rA+x),\qquad\mbox{for $A\subset\R^d$.}$$
One says that 
$\nu$ is a tangent measure of $\mu$ at $x$ if $\nu$ is a non-zero Radon measure and
there are a sequences $\{r_k\}_k$, $\{c_k\}_k$
of positive numbers with $\lim_{k\to\infty}r_k=0$
such that the measures
$c_k\,T_{x,r_k}\#\mu$ converge weakly to $\nu$ as $k\to\infty$. 

A measure $\nu$ in $\R^d$ is called $n$-flat if there exists an $n$-plane $L\subset\R^d$
and $c>0$ such that $\nu=c\,\HH^n|_L$. On the other hand, $\nu$ is called $n$-uniform
if it is a non-zero measure and there exists $c>0$ such that
$$\nu(B(x,r))=c\,r^n\quad \mbox{for all $x\in\supp(\nu)$, $r>0$.}$$

By the so called Marstrand-Mattila rectifiability
criterion (see Theorem 16.7 of \cite{Mattila-llibre}) if $\mu$ 
is a finite Borel measure in $\R^d$ such that, for $\mu$-a.e.\ $x\in\R^d$, $0< \Theta_*^n(x,\mu)\leq \Theta^{n,*}(x,\mu)<\infty$, it turns out that $\mu$ is $n$-rectifiable if and only
if, for $\mu$-a.e.\ $x\in\R^d$, all tangent measures at $x$ are $n$-flat. We will apply
this criterion to prove Proposition \ref{propoclau}.
The first step consists in showing that for $\mu$-a.e.\ $x\in\R^d$ all tangent
measures at $x$ are $n$-uniform:

\begin{lemma}\label{lemlem}
Let $\mu$ be a finite Borel measure in $\R^d$ such that, for $\mu$-a.e.\ $x\in\R^d$, $0< \Theta_*^n(x,\mu)\leq \Theta^{n,*}(x,\mu)<\infty$ and moreover
\begin{equation}\label{eqaaa209}
\lim_{r\to0} \Delta_{\mu,\vphi}(x,r)=0.
\end{equation}
Then all tangent measures of $\mu$ at $x$ are $n$-uniform for $\mu$-a.e.\ $x\in\R^d$.
\end{lemma}

\begin{proof}
We will show that for $\mu$ almost all $x\in\R^d$, any tangent measure $\nu$ at $x$
is $n$-AD-regular and satisfies 
\begin{equation}\label{eqaaa210}
\Delta_{\nu,\vphi}(x,r)=0\quad \mbox{ for all $x\in\supp(\nu)$ and all $r>0$.}
\end{equation}
By Theorem 3.10 from \cite{CGLT}, this implies that $\nu$ is $n$-uniform. 

The $n$-AD-regularity of any tangent measure $\nu$ at $x$, for $\mu$-a.e.\ 
$x\in\R^d$, follows from the fact that $0< \Theta_*^n(x,\mu)\leq \Theta^{n,*}(x,\mu)<\infty$
$\mu$-a.e., as shown in Lemma 14.7(1) of \cite{Mattila-llibre}.

We turn now our attention to \rf{eqaaa210}. To prove this we follow the same approach of 
\cite[Lemma 20.7]{Mattila-llibre} in connection with the existence principal values for singular integrals. 
For every $\ve>0$, by Egoroff's
theorem we can find a compact set $F$ with $\mu(\R^d\setminus F)<\ve$ where
the convergence \rf{eqaaa209} is uniform. Moreover, $F$ can be chosen so that
$0< \Theta_*^n(x,\mu)\leq \Theta^{n,*}(x,\mu)<\infty$ for all $x\in F$. 

Let $x\in F$ be a $\mu$-density point of $F$. That is,
$$\lim_{r\to 0}\frac{\mu(B(x,r)\setminus F)}{\mu(B(x,r))}=0.$$
We claim that if $\nu$ is a tangent measure of $\mu$ at $x$, then \rf{eqaaa210} holds. 
 To see this, take  sequences
$c_k,r_k>0$ such that $c_k\,T_{x,r_k}\#\mu$ converges weakly to $\nu$.
By
\cite[Remark 14.4(1)]{Mattila-llibre} we may assume that $c_k=1/\mu(B(x,r_k))$.
Moreover, we can take a sequence of points $x_k\in F$ such that
$$z_k = \frac{x_k-x}{r_k}\to z\quad\mbox{ as $k\to\infty$.}$$
This is also shown in the proof of  \cite[Lemma 14.7(1)]{Mattila-llibre}.

Notice that
$$\int \bigl(\vphi_r(z-y) - \vphi_{2r}(z-y)\bigr)\,d\nu(y)
= \lim_{k\to\infty} \frac1{\mu(B(x,r_k))} \int \bigl(\vphi_r(z_k-y) - \vphi_{2r}(z_k-y)\bigr)\,dT_{x,r_k}\#\mu(y).$$
This follows from the weak convergence of $\frac1{\mu(B(x,r_k))} T_{x,r_k}\#\mu$ to $\nu$ and 
the uniform convergence
$$\bigl(\vphi_r(z_k-\cdot) - \vphi_{2r}(z_k-\cdot)\bigr) \to \bigl(\vphi_r(z-\cdot) - \vphi_{2r}(z-\cdot)\bigr)
\quad\mbox{ as $k\to\infty$,}$$
taking also into account the fast decay at $\infty$ of $\vphi_r$ and $\vphi_{2r}$.
Then we have
\begin{align*}
\Delta_{\nu,\vphi}(x,r) & = \lim_{k\to\infty}\frac1{\mu(B(x,r_k))}  \int \bigl(\vphi_r(z_k-y) - \vphi_{2r}(z_k-y)\bigr)\,dT_{x,r_k}\#\mu(y)\\
&=\lim_{k\to\infty}\frac1{\mu(B(x,r_k))} \int \bigl(\vphi_r(z_k-T_{x,r_k}(y)) - \vphi_{2r}(z_k-
T_{x,r_k}(y))\bigr)\,d\mu(y)\\
& = \lim_{k\to\infty}\frac1{\mu(B(x,r_k))} \int \Bigl(\vphi_r\Bigl(\frac{x_k-y}{r_k}\Bigr) - \vphi_{2r}\Bigl(\frac{x_k-y}{r_k}\Bigr) \Bigr)\,d\mu(y)\\
&= \lim_{k\to\infty}\frac{r_k^n}{\mu(B(x,r_k))} \int \bigl(\vphi_{r\,r_k}(x_k-y) - \vphi_{2\,r\,r_k}(x_k-y) \bigr)\,d\mu(y).
\end{align*}
Recalling that $ \Theta_*^n(x,\mu)>0$ and using the uniform convergence of \rf{eqaaa209} in $F$, we infer that the last limit vanishes, and so $\Delta_{\nu,\vphi}(x,r)=0$, as wished.
\end{proof}
\vv

\begin{propo}\label{propo2}
Let $\mu$ be a Radon measure in $\R^d$ such that, for $\mu$-a.e.\ $x\in\R^d$, $0< \Theta_*^n(x,\mu)\leq \Theta^{n,*}(x,\mu)<\infty$. If for 
 $\mu$-a.e.\ $x\in\R^d$ all tangent measures of $\mu$ at $x$ are $n$-uniform, then $\mu$ is
 rectifiable.
 \end{propo}

Clearly, this result in conjunction with Lemma \ref{lemlem} proves Proposition \ref{propoclau}
and concludes the proof of (b)$\Rightarrow$(a) from Theorem \ref{teomain}.

\begin{proof}
As remarked above, by the Marstrand-Mattila criterion, it is enough to show that 
for 
 $\mu$-a.e.\ $x\in\R^d$ all tangent measures of $\mu$ at $x$ are $n$-flat.
To prove the proposition we will take into account that, for $\mu$-a.e.\ $x\in\R^d$, if $\nu$ is tangent measure
of $\mu$ at $x$, then any tangent measure $\sigma$ of $\nu$ at any point $y\in\supp(\nu)$ is also a tangent
measure of $\mu$ at $x$. See Theorem 14.16 from \cite{Mattila-llibre}.

By a result of Kirchheim and Preiss \cite{KiP} it follows that any $n$-uniform measure is supported on an
$n$-dimensional real analytic variety. In particular, it turns out that it is $n$-rectifiable and thus
has flat tangent measures at some points in its support. Thus we deduce that for $\mu$-a.e.\ $x\in\R^d$ 
there exists flat tangent measures to $\mu$ at $x$. 

To summarize, we know that for $\mu$-a.e.\ $x\in\R^d$ all tangent measures at $x$ are $n$-uniform and at least one of the tangent measures is $n$-flat. It is shown in \cite{Preiss} that this implies that
all tangent measures at $x$ are flat. Indeed this is one of the key ingredients of the proof
of Preiss' theorem, which has been stated in \rf{teopreiss}. See also 
Theorem 6.10 of the nice monograph by De Lellis \cite{DeL} for a very transparent argument. So the proposition follows.
\end{proof}


\section{Proof of {\rm (c)$\Leftrightarrow$(a)} from Theorem \ref{teomain}}

Let $\mu$ be a Radon measure in $\R^d$ such that, for $\mu$-a.e.\ $x\in\R^d$, $0< \Theta_*^n(x,\mu)\leq \Theta^{n,*}(x,\mu)<\infty$. 
We will show in this section that $\mu$ is $n$-rectifiable if and only if
\begin{equation}\label{eqaaa250}
\lim_{r\to0}\left|\frac{\mu(B(x,r))}{r^n} - \frac{\mu(B(x,2r))}{(2r)^n}\right| =0
\quad \mbox{ for $\mu$-a.e.\ $x\in\R^d$.}
\end{equation}

Notice first that if $\mu$ is $n$-rectifiable, then the density $\Theta^n(x,\mu)$ exists $\mu$-a.e.\ and thus
\rf{eqaaa250} holds.

To prove the converse implication we may assume that $\mu$ is compactly supported and thus finite. By Proposition
\ref{propoclau} it is enough to show that 
$$\lim_{r\to0} \Delta_{\mu,\vphi}(x,r)=0.$$
To this end, recall that $\Delta_{\mu,\vphi}(x,r)$ can be written as a convex combination of the
terms of $\Delta_\mu(x,s)$, $s>0$. Indeed, as shown in the proof of Corollary 3.12 from \cite{CGLT},
we have
$$\Delta_{\mu,\vphi}(x,r) = \int_0^\infty \Delta_\mu(x,s)\,\wt \vphi_r(s)\,ds,$$
where
$$\wt\vphi_r(s) = \frac{2\,s^{n+1}}{r^{n+2}}\,e^{-s^2/r^2}.$$
For $\lambda\geq0$, we split the integral as follows:
$$
\bigl|\Delta_{\mu,\vphi}(x,r)\bigr| \leq \int_0^{\lambda\,r} \bigl|\Delta_\mu(x,s)\bigr|\,\wt \vphi_r(s)\,ds
+ \int_{\lambda\,r}^\infty \bigl|\Delta_\mu(x,s)\bigr|\,\wt \vphi_r(s)\,ds.$$
We have
$$\int_{\lambda\,r}^\infty \wt \vphi_r(s)\,ds =
\int_{\lambda\,r}^\infty \frac{2\,s^{n+1}}{r^{n+2}}\,e^{-s^2/r^2}\,ds = 
\int_{\lambda}^\infty 2\,t^{n+1}\,e^{-t^2}\,dt.$$
Therefore, for $0<r<1$, choosing $\lambda = r^{-1/2}$ we get
$$\bigl|\Delta_{\mu,\vphi}(x,r)\bigr| \leq c\,\sup_{0<s\leq r^{1/2}}\bigl|\Delta_\mu(x,s)\bigr|
+ \sup_{s>0}\bigl|\Delta_\mu(x,s)\bigr|\,
\int_{r^{-1/2}}^\infty 2\,t^{n+1}\,e^{-t^2}\,dt.$$
The first term on the right side tends to $0$ as $r\to0$ because of \rf{eqaaa250}, 
and the second too because $\sup_{s>0}\bigl|\Delta_\mu(x,s)\bigr|<\infty$, taking into account
that $\Theta^{n,*}(x,\mu)<\infty$ and $\|\mu\|<\infty$, by our assumption.

\vvv

\end{document}